\documentclass[numbook,envcountsame,smallcondensed,draft]{amsart}

\usepackage{amssymb,amsmath,amsfonts,mathrsfs,cite,color}

\newcommand{\word}{\mathbf}
\newcommand{\wa}{\word{a}}
\newcommand{\wb}{\word{b}}
\newcommand{\we}{\word{e}}
\newcommand{\wh}{\word{h}}
\newcommand{\ws}{\word{s}}
\newcommand{\wt}{\word{t}}
\newcommand{\wu}{\word{u}}
\newcommand{\wv}{\word{v}}
\newcommand{\ww}{\word{w}}

\newcommand{\variety}{\mathbf}
\newcommand{\vA}{\variety{A}}
\newcommand{\vC}{\variety{C}}
\newcommand{\vD}{\variety{D}}
\newcommand{\vE}{\variety{E}}
\newcommand{\vF}{\variety{F}}
\newcommand{\vH}{\variety{H}}
\newcommand{\vJ}{\variety{J}}
\newcommand{\vK}{\variety{K}}
\newcommand{\vL}{\variety{L}}
\newcommand{\vM}{\variety{M}}
\newcommand{\vN}{\variety{N}}
\newcommand{\vO}{\variety{O}}
\newcommand{\vP}{\variety{P}}
\newcommand{\vQ}{\variety{Q}}
\newcommand{\vR}{\variety{R}}
\newcommand{\vSL}{\variety{SL}}
\newcommand{\vT}{\variety{T}}
\newcommand{\vV}{\variety{V}}
\newcommand{\vX}{\variety{X}}

\newcommand{\setofwords}{\mathscr}
\newcommand{\sX}{\setofwords{X}}

\newcommand{\Acen}{\variety{A}_\mathsf{cen}}
\newcommand{\Acom}{\variety{A}_\mathsf{com}}

\DeclareMathOperator{\con}{con}
\DeclareMathOperator{\simple}{sim}
\DeclareMathOperator{\mul}{mul}
\DeclareMathOperator{\occ}{occ}
\DeclareMathOperator{\var}{var}

\newtheorem{theorem}{Theorem}[section]
\newtheorem{proposition}[theorem]{Proposition}
\newtheorem{lemma}[theorem]{Lemma}
\newtheorem{corollary}[theorem]{Corollary}
\newtheorem{observation}[theorem]{Observation}
\theoremstyle{definition}

\numberwithin{equation}{section}

\makeatletter

\renewcommand*\subjclass[2][2010]{\def\@subjclass{#2}\@ifundefined{subjclassname@#1}{\ClassWarning{\@classname}{Unknown edition (#1) of Mathematics Subject Classification; using '2010'.}}{\@xp\let\@xp\subjclassname\csname subjclassname@#1\endcsname}}

\renewcommand{\subjclassname}{\textup{2010} Mathematics Subject Classification}

\makeatother

\begin{document}

\title[Cross varieties of aperiodic monoids with commuting idempotents]{Cross varieties of aperiodic monoids\\ with commuting idempotents}
\thanks{The work is supported by the Ministry of Science and Higher Education of the Russian Federation (project FEUZ-2020-0016).}

\author{S.V. Gusev}

\address{Ural Federal University, Institute of Natural Sciences and Mathematics, Lenina 51, 620000 Ekaterinburg, Russia}

\email{sergey.gusb@gmail.com}

\begin{abstract}
A variety of algebras is called Cross if it is finitely based, finitely generated, and has finitely many subvarieties. In present article, we classify all Cross varieties of aperiodic monoids with commuting idempotents. 
\end{abstract}

\keywords{Monoid, aperiodic monoid, commuting idempotents, variety, Cross variety}

\subjclass{20M07}

\maketitle

\section{Introduction}
\label{sec: introduction}

A variety of algebras is called \textit{finitely based} if it has a finite basis of its identities. A variety is called \textit{finitely generated} if it is generated by a finite algebra. A variety with finitely many subvarieties is said to be \textit{small}. A finitely generated, finitely based, small variety of algebras is called a \textit{Cross variety}. Finite members from
several classical classes of algebras such as groups and associative rings generate Cross varieties (see~\cite{Oates-Powell-64} and~\cite{Kruse-73,Lvov-73}, respectively). But this result does not hold in general. For example, there are a lot of finite semigroups or monoids that generate non-Cross varieties.

For any class of algebras, one method of describing Cross varieties in this class is to find its minimal non-Cross varieties, or \textit{almost Cross varieties}. Any non-Cross variety contains some almost Cross subvariety by Zorn's lemma. It follows that a variety is Cross if and only if it does not contain any almost Cross variety. So, if one manages to classify all almost Cross varieties within some class of varieties, then this classification implies a description of all Cross varieties in this class.

The present article is concerned with the class of \textit{aperiodic monoids}, i.e., monoids that have trivial subgroups only. For a long time, it was known only two explicit examples of almost Cross varieties of monoids: the variety of all commutative monoids~\cite{Head-68} and the variety of all idempotent monoids~\cite{Wismath-86}. The second of these varieties is the first example of an almost Cross variety of aperiodic monoids. More recently, Jackson~\cite{Jackson-05} found two new examples of almost Cross monoid varieties $\vM$ and $\vN$. It turned out that $\vM$ and $\vN$ are finitely generated subvarieties of the class $\Acen$ of aperiodic monoids with central idempotents. In~\cite{Lee-11}, Lee showed that $\vM$ and $\vN$ are the unique finitely generated almost Cross subvarieties of $\Acen$. Further, Lee~\cite{Lee-13} found one more (non-finitely generated) almost Cross subvariety $\vL$ of $\Acen$ and established that only $\vL$, $\vM$ and $\vN$ are almost Cross subvarieties of $\Acen$. 

It is of fundamental interest to generalize this result by Lee from $\Acen$ to some larger class of monoids. The class $\Acom$ of aperiodic monoids with commuting idempotents is a natural candidate. The first step in describing almost Cross subvarieties of $\Acom$ was taken by Lee~\cite{Lee-14}. Namely, he proved that a subvariety of $\Acom$ that satisfies the identity $xyx^2\approx x^2yx$ is Cross if and only if it excludes the almost Cross varieties $\vL$, $\vM$ and $\vN$. However, there exist almost Cross subvarieties of $\Acom$ that
are different from $\vL$, $\vM$ and $\vN$. The two of such varieties $\vK$ and $\overleftarrow{\vK}$ were exhibited in~\cite{Gusev-Vernikov-18}. Finally, just recently, two more almost Cross subvarieties $\vJ$ and $\overleftarrow{\vJ}$ of $\Acom$ were found in~\cite{Gusev-20}. In present article, we provide two new examples of almost Cross subvarieties $\vP$ and $\overleftarrow{\vP}$ of $\Acom$ and then completely classify all Cross varieties within the class $\Acom$.

The article consists of five sections. Some background results are first given in Section~\ref{sec: prelim}. In Section~\ref{sec: var P}, we provide two new examples of almost Cross subvarieties of $\Acom$. Let $\vO$ denote the monoid variety given by the identities 
\begin{align}
\label{xtyzxy=xtyzyx}
xtyzxy&\approx xtyzyx,\\
\label{xtxyzy=xtyxzy}
xtxyzy&\approx xtyxzy.
\end{align}
The subvarieties of $\vO$ have been actively studied. In particular, any subvariety of $\vO$ is finitely based and $\vO$ is a maximal variety with such a property~\cite{Lee-12}. Section~\ref{sec: var O} is devoted to the description of aperiodic almost Cross subvarieties of $\vO$. Finally, in Section~\ref{sec: main result}, the main result of the article is formulated and is proved. Namely, it is established that a subvariety of $\Acom$ is Cross if and only if it excludes the nine almost Cross subvarieties of $\Acom$.

\section{Preliminaries} 
\label{sec: prelim}

\subsection{Locally finite varieties}

Recall that a variety of universal algebras is called \textit{locally finite} if all of its finitely generated members are finite.

\begin{lemma}[\!{\cite[Lemma~2.1]{Jackson-Lee-18}}] 
\label{L: LF+small=FG}
A locally finite variety $\vV$ of algebras is finitely generated if and only if there exists no strictly increasing infinite chain $\vV_1 \subset \vV_2 \subset \vV_3 \subset \cdots$ of varieties such that $\vV=\bigvee_{i\ge1}\vV_i$. Consequently, any locally finite, small variety is finitely generated.\qed
\end{lemma}

The following assertion directly follows from~\cite[Proposition~3.1]{Sapir-87}.

\begin{lemma}
\label{L: xyx=xyxx is locally finite}
Each monoid variety that satisfies the identity
\begin{equation}
\label{xyx=xyxx}
xyx\approx xyx^2
\end{equation}
is locally finite.\qed
\end{lemma}

\subsection{Subvarieties of $\Acom$}

For each $n\ge1$, let $\vA_n$ denote the variety defined by the identities
\begin{align}
\label{x^n=x^{n+1}}
x^n&\approx x^{n+1},\\
\label{x^ny^n=y^nx^n}
x^ny^n&\approx y^nx^n.
\end{align}
It is easily shown that the inclusions $\vA_1 \subset \vA_2 \subset \cdots \subset \Acom$ hold and are proper. The following claim is evident.

\begin{observation}
\label{O: var in Acom}
The class $\Acom$ is not a variety, but each of its subvarieties is contained in $\vA_n$ for all sufficiently large $n$.\qed
\end{observation}

A variety of monoids is called \textit{completely regular} if it consists of \textit{completely regular monoids}~(i.e., unions of groups). Let $\vT$ and $\vSL$ denote the trivial variety of monoids and the variety of all semilattice monoids, respectively. We note that $\vA_1=\vSL$.

\begin{observation}
\label{O: cr in Acom}
If $\vV$ is a completely regular subvariety of $\Acom$, then $\vV\in\{\vT,\vSL\}$.
\end{observation}

\begin{proof}
Observation~\ref{O: var in Acom} implies that $\vV$ satisfies the identities~\eqref{x^n=x^{n+1}} and~\eqref{x^ny^n=y^nx^n} for some $n\ge1$. If $n$ is the least number with such a property, then $n=1$ because $\vV$ is completely regular. Hence $\vV\subseteq\vSL$. It is well known that $\vSL$ covers $\vT$, whence $\vV\in\{\vT,\vSL\}$.
\end{proof}

\begin{observation}
\label{O: cr and comm are Cross}
Each commutative or completely regular subvariety of $\Acom$ is Cross.
\end{observation}

\begin{proof}
A completely regular subvariety of $\Acom$ is either $\vT$ or $\vSL$ by Observation~\ref{O: cr in Acom} and so Cross. A commutative aperiodic variety is Cross by the results of~\cite{Head-68}. 
\end{proof}

\subsection{Words, identities, Rees quotients of free monoids}

Let~$\sX^\ast$ denote the free monoid over a countably infinite alphabet~$\sX$.
Elements of~$\sX$ are called \textit{letters} and elements of~$\sX^\ast$ are called \textit{words}. We treat the identity element of~$\sX$ as \textit{the empty word}. The \textit{content} of a word $\ww$, i.e., the set of all letters occurring in $\ww$, is denoted by $\con(\ww)$. An identity is an expression $\wu\approx\wv$ where $\wu,\wv\in\sX^\ast$.

\begin{observation}
\label{O: equivalent identities}
Let $\wa,\wa',\wb,\wb'\in\sX^\ast$, $x,t_1,t_2,\dots,t_k\in\sX$ and $\vV$ be a monoid variety that satisfies the identity~\eqref{x^n=x^{n+1}}. Suppose that $t_1,t_2,\dots,t_k\notin\con(\wa\wa'\wb\wb')$. Then the identities 
$$
\sigma:\enskip\wa t_1x^n\wb\approx\wa' t_1x^n\wb'\ \text{ and }\ 
\tau:\enskip\wa \prod_{i=1}^k(t_ix^n)\wb\approx \wa' \prod_{i=1}^k(t_ix^n)\wb'
$$ 
are equivalent in $\vV$.
\end{observation}

\begin{proof}
We substitute $x^n\prod_{i=2}^k(t_ix^n)$ for $t_1$ in the identity $\sigma$ and obtain the identity $\tau$. So, $\tau$ follows from $\sigma$.

Conversely, substitute~1 for the letters $t_2,t_3,\dots, t_k$ in the identity $\tau$. We get the identity $\wa t_1x^{kn}\wb\approx\wa' t_1x^{kn}\wb'$. Evidently, this identity together with~\eqref{x^n=x^{n+1}} implies $\sigma$. Therefore, $\sigma$ and $\tau$ are equivalent in $\vV$.
\end{proof}

For any set of words $W\subseteq\sX^\ast$, let $S(W)$ denote the Rees quotient over the ideal of $\sX^\ast$ consisting of all words that are not subwords of words in $W$.
A word~$\ww$ is an \textit{isoterm} for a variety~$\vV$ if it violates any non-trivial identity of the form $\ww\approx\ww'$. The following statement shows it is relatively easy to check if $S(W)$ is contained in some given variety.

\begin{lemma}[\!{\cite[Lemma~3.3]{Jackson-05}}] 
\label{L: isoterm}
For any variety~$\vV$ and any set of words~$W$, the inclusion $S(W)\in\vV$ holds if and only if all words from $W$ are isoterms for~$\vV$.\qed
\end{lemma}

\subsection{Decompositions of words}
\label{subsec: decompositions of words}

A letter is called \textit{simple} [\textit{multiple}] \textit{in a word} $\ww$ if it occurs in $\ww$ once [at least twice]. The set of all simple [multiple] letters in a word $\ww$ is denoted by $\simple(\ww)$ [respectively, $\mul(\ww)$]. The number of occurrences of the letter $x$ in $\ww$ is denoted by $\occ_x(\ww)$. For a word $\ww$ and letters $x_1,x_2,\dots,x_k\in \con(\ww)$, let $\ww(x_1,x_2,\dots,x_k)$ denote the word obtained from $\ww$ by deleting all letters except $x_1,x_2,\dots,x_k$. 

Let $\ww$ be a word and $\simple(\ww)=\{t_1,t_2,\dots, t_m\}$. We may assume without loss of generality that $\ww(t_1,t_2,\dots, t_m)=t_1t_2\cdots t_m$. Then $\ww = t_0\ww_0 t_1 \ww_1 \cdots t_m \ww_m$ where $\ww_0,\ww_1,\dots,\ww_m$ are possibly empty words and $t_0$ is the empty word. The words $\ww_0$, $\ww_1$, \dots, $\ww_m$ are called \textit{blocks} of $\bf w$, while $t_0,t_1,\dots,t_m$ are said to be \textit{dividers} of $\ww$. The representation of the word $\ww$ as a product of alternating dividers and blocks starting with the divider $t_0$ and ending with the block $\ww_m$ is called a \textit{decomposition} of the word $\ww$. For a given word $\ww$, a letter $x\in\con(\ww)$ and a natural number $i\le\occ_x(\ww)$, we denote by $h_i(\ww,x)$ the right-most divider of $\ww$ that precedes the $i$th occurrence of $x$ in $\ww$.{\sloppy

}
\begin{lemma}[\!{\cite[Proposition~2.13 and Lemma~2.14]{Gusev-Vernikov-18}}] 
\label{L: equivalent decompositions}
Let $\vV$ be a non-completely regular non-commutative monoid variety and $\wu\approx \wv$ be an identity that holds in $\vV$. Suppose that 
\begin{equation}
\label{decomposition of u}
t_0\wu_0 t_1 \wu_1 \cdots t_m \wu_m
\end{equation}
is the decomposition of $\wu$. Then the decomposition of $\wv$ has the form
\begin{equation}
\label{decomposition of v}
t_0\wv_0 t_1 \wv_1 \cdots t_m \wv_m
\end{equation}
for some words $\wv_0,\wv_1,\dots,\wv_m$ and $\con(\wu_0\wu_1\cdots\wu_m)=\con(\wv_0\wv_1\cdots\wv_m)$.\qed
\end{lemma}

\subsection{The variety $\vF$}
\label{subsec: var F}

Let $\vF$ denote the monoid variety given by the identities~\eqref{xyx=xyxx} and
\begin{align}
\label{xxy=xxyx}
x^2y&\approx x^2yx,\\
\label{xxyy=yyxx} 
x^2y^2&\approx y^2x^2,\\
\label{xyzxy=yxzxy} 
xyzxy&\approx yxzxy.
\end{align}

\begin{lemma}[\!\mdseries{\cite[Proposition~6.9(i)]{Gusev-Vernikov-18}}]
\label{L: word problem F}
A non-trivial identity $\wu\approx\wv$ holds in the variety $\vF$ if and only if $\simple(\wu)=\simple(\wv)$, $\mul(\wu)=\mul(\wv)$ and $h_1(\wu,x)= h_1(\wv,x)$, $h_2(\wu,x)= h_2(\wv,x)$ for all $x\in \con(\wu)$.\qed
\end{lemma}

If $\wu,\wv\in\sX^\ast$ and $\Sigma$ is an identity or an identity system, then we will write $\wu\stackrel{\Sigma}\approx\wv$ in the case when the identity $\wu\approx\wv$ follows from $\Sigma$.

\begin{lemma}
\label{L: does not contain F}
Let $\vV$ be a non-completely regular monoid variety that satisfies the identity~\eqref{x^n=x^{n+1}} for some $n\ge2$. Then $\vF\nsubseteq\vV$ if and only if $\vV$ satisfies the identity
\begin{equation}
\label{xyx^n=x^nyx^n}
xyx^n\approx x^nyx^n.
\end{equation}
\end{lemma}

\begin{proof}
The variety $\vF$ violates the identity~\eqref{xyx^n=x^nyx^n} for any $n\ge2$ by Lemma~\ref{L: word problem F}. Therefore, if $\vV$ satisfies~\eqref{xyx^n=x^nyx^n}, then $\vF\nsubseteq\vV$.

Conversely, suppose that $\vF\nsubseteq\vV$. If $\vV$ is commutative, then it satisfies the identity~\eqref{xyx^n=x^nyx^n} because 
$$
x^nyx^n\stackrel{\text{comm.}}\approx xyx^{2n-1}\stackrel{\eqref{x^n=x^{n+1}}}\approx xyx^n.
$$
Hence it suffices to assume that the variety $\vV$ is non-commutative. By assumption, there exists an identity $\wu \approx \wv$ that is satisfied by the variety $\vV$ but not by the variety $\vF$. Let~\eqref{decomposition of u} be the decomposition of $\wu$. In view of Lemma~\ref{L: equivalent decompositions}, the decomposition of $\wv$ has the form~\eqref{decomposition of v} and $\mul(\wu)=\mul(\wv)$. Lemma~\ref{L: word problem F} implies that there are a letter $x\in\mul(\wu)$ and $k\in\{1,2\}$ such that $h_k(\wu,x)\ne h_k(\wv, x)$. If $k=1$ then we multiply the identity $\wu\approx\wv$ by $xt$ on the left where $t\notin\con(\wu)$. So, we may assume that $k=2$. Suppose that $h_2(\wu,x)=t_i$ and $h_2(\wv,x)=t_j$ where $i\ne j$. We may assume without any loss that $i>j$. Then $\vV$ satisfies the identity
$$
\wu(x,t_i)=xt_ix^p\approx x^qt_ix^r=\wv(x,t_i)
$$
where $p\ge1$, $q\ge 2$ and $r\ge 0$. Now we multiply this identity by $x^n$ on the right. Taking into account the identity~\eqref{x^n=x^{n+1}}, we get that $\vV$ satisfies the identity
\begin{equation}
\label{xyx^n=x^qyx^n}
xyx^n\approx x^qyx^n.
\end{equation}
Since
$$
xyx^n\stackrel{\eqref{xyx^n=x^qyx^n}}\approx x^{1+n(q-1)}yx^n\stackrel{\eqref{x^n=x^{n+1}}}\approx x^nyx^n,
$$
the variety $\vV$ satisfies the identity~\eqref{xyx^n=x^nyx^n}.
\end{proof}

\subsection{The variety $\vQ$}
\label{subsec: var Q}

For an identity system $\Sigma$, we denote by $\var\,\Sigma$ the variety of monoids given by $\Sigma$. Put
$$
\vQ=\var\{\eqref{xyx=xyxx},\,\eqref{xxyy=yyxx},\, xyx^2\approx x^2yx^2\}.
$$

\begin{lemma}[\!{\cite[Lemma~5.1]{Lee-14}}] 
\label{L: word problem Q}
An identity of the form
\begin{equation}
\label{u_0t_0...t_mu_m=v_0t_0...t_mv_m}
\wu_0 \prod_{i=1}^m (t_i\wu_i) \approx \wv_0 \prod_{i=1}^m (t_i\wv_i)
\end{equation}
where $\wu_0, \wv_0, \wu_1,\wv_1,\dots, \wu_m, \wv_m$ are words not containing any of the letters $t_1,t_2,\dots, t_m$ holds in $\vQ$ if and only if $\con(\wu_i)=\con(\wv_i)$ for any $i=0,1,\dots,m$.\qed{\sloppy

}
\end{lemma}

In fact, the proof of the following assertion is similar to the proof of Lemma~5.3 in~\cite{Lee-14}. We provide its proof for the sake of completeness.

\begin{lemma}
\label{L: does not contain Q}
Let $\vV$ be a non-completely regular monoid variety that satisfies the identity~\eqref{x^n=x^{n+1}} for some $n\ge2$. Then $\vQ\nsubseteq\vV$ if and only if $\vV$ satisfies the identity
\begin{equation}
\label{x^nyzx^n=x^nyxzx^n}
x^nyzx^n\approx x^nyxzx^n.
\end{equation}
\end{lemma}

\begin{proof}
The variety $\vQ$ violates the identity~\eqref{x^nyzx^n=x^nyxzx^n} for any $n\ge1$ by Lemma~\ref{L: word problem Q}. 
Therefore, if $\vV$ satisfies~\eqref{x^nyzx^n=x^nyxzx^n}, then $\vQ\nsubseteq\vV$.

Conversely, suppose that $\vQ\nsubseteq\vV$. If $\vV$ is commutative, then it satisfies the identity~\eqref{x^nyzx^n=x^nyxzx^n} because 
$$
x^nyxzx^n\stackrel{\text{comm.}}\approx x^{n+1}yzx^n\stackrel{\eqref{x^n=x^{n+1}}}\approx x^nyzx^n.
$$
Hence it suffices to assume that the variety $\vV$ is non-commutative. 
By assumption, there exists an identity $\wu \approx \wv$ that is satisfied by the variety $\vV$ but not by the variety $\vQ$. Let~\eqref{decomposition of u} be the decomposition of $\wu$. In view of Lemma~\ref{L: equivalent decompositions}, the decomposition of $\wv$ has the form~\eqref{decomposition of v}. According to Lemma~\ref{L: word problem Q}, $\con(\wu_i)\ne\con(\wv_i)$ for some
$i\in\{0,1,\dots,m\}$, say $x\in\con(\wu_i)\setminus\con(\wv_i)$.

Evidently, the identity $\wu \approx \wv$ implies the identity $xh_1\wu h_2x \approx xh_1\wv h_2x$ where $h_1,h_2\notin\con(\wu)=\con(\wv)$. The last identity implies the identity
\begin{equation}
\label{x^pt_ix^qt_{i+1}x^r=x^st_it_{i+1}x^t}
x^pt_ix^qt_{i+1}x^r \approx x^st_it_{i+1}x^t
\end{equation}
for some $p,q,r,s,t\ge1$. Then the identity
\begin{equation}
\label{x^nyx^nzx^n=x^nyzx^n}
x^nyx^nzx^n \approx x^nyzx^n.
\end{equation}
holds in $\vV$ because
$$
x^nyx^nzx^n \stackrel{\eqref{x^n=x^{n+1}}}\approx (x^n)^py(x^n)^qz(x^n)^r\stackrel{\eqref{x^pt_ix^qt_{i+1}x^r=x^st_it_{i+1}x^t}}\approx (x^n)^syz(x^n)^t \stackrel{\eqref{x^n=x^{n+1}}}\approx x^nyzx^n.
$$
Finally, since
$$
x^nyzx^n \stackrel{\eqref{x^nyx^nzx^n=x^nyzx^n}}\approx x^nyx^nzx^n \stackrel{\eqref{x^n=x^{n+1}}}\approx x^nyx^{n+1}zx^n \stackrel{\eqref{x^nyx^nzx^n=x^nyzx^n}}\approx x^nyxzx^n,
$$
the variety $\vV$ satisfies the identity~\eqref{x^nyzx^n=x^nyxzx^n}.
\end{proof}

A word $\ww$ is said to be $n$-\textit{limited} if $\occ_x(\ww)\le n$ for any $x\in\con(\ww)$.

\begin{corollary}
\label{C: w is equivalent to (2n+2)-limited word}
Let $\ww$ be a word and $\vV$ a monoid variety that satisfies the identities~\eqref{x^n=x^{n+1}} and
$$
\kappa_{n,j}:\enskip \prod_{i=0}^n(t_ix)\approx \Bigl(\prod_{i=0}^{j-1}(t_ix)\Bigr) \bigl(t_jx^2\bigr) \Bigl(\prod_{i=j+1}^n(t_ix)\Bigr)
$$
for some $n\ge1$ and $0\le j\le n$. If $\vQ\nsubseteq\vV$, then there is a $(2n+2)$-limited word $\ww'$ such that $\ww\approx\ww'$ is satisfied by $\vV$.
\end{corollary}

\begin{proof}
In view of Lemma~\ref{L: does not contain Q}, $\vV$ satisfies~\eqref{x^nyzx^n=x^nyxzx^n}. It follows that the identities
$$
\begin{aligned}
&x\Bigl(\prod_{i=1}^n(t_ix)\Bigr)\,yz\,x\Bigl(\prod_{i=1}^n(h_ix)\Bigr)\\
\stackrel{\kappa_{n,j}}\approx {}&x\Bigl(\prod_{i=1}^{j-1}(t_ix)\Bigr) \bigl(t_jx^n\bigr) \Bigl(\prod_{i=j+1}^n(t_ix)\Bigr)\,yz\,x\Bigl(\prod_{i=1}^{j-1}(h_ix)\Bigr) \bigl(h_jx^n\bigr) \Bigl(\prod_{i=j+1}^n(h_ix)\Bigr)\\
\stackrel{\eqref{x^nyzx^n=x^nyxzx^n}}\approx {}&x\Bigl(\prod_{i=1}^{j-1}(t_ix)\Bigr) \bigl(t_jx^n\bigr) \Bigl(\prod_{i=j+1}^n(t_ix)\Bigr)\,yxz\,x\Bigl(\prod_{i=1}^{j-1}(h_ix)\Bigr) \bigl(h_jx^n\bigr) \Bigl(\prod_{i=j+1}^n(h_ix)\Bigr)\\
\stackrel{\kappa_{n,j}}\approx {}&x\Bigl(\prod_{i=1}^n(t_ix)\Bigr)\,yxz\,x\Bigl(\prod_{i=1}^n(h_ix)\Bigr)
\end{aligned}
$$
hold in $\vV$. We see that $\vV$ satisfies the identity
\begin{equation}
\label{..yz..=yxz..}
x\Bigl(\prod_{i=1}^n(t_ix)\Bigr)\,yz\,x\Bigl(\prod_{i=1}^n(h_ix)\Bigr)\approx x\Bigl(\prod_{i=1}^n(t_ix)\Bigr)\,yxz\,x\Bigl(\prod_{i=1}^n(h_ix)\Bigr).
\end{equation}
We note that the identity~\eqref{..yz..=yxz..} allows us to delete the $(n+2)$th occurrence of some letter $x$ in a word $\ww$ whenever $\occ_x(\ww)>2n+2$. This implies that, for any $\ww\in\sX^\ast$, there is a $(2n+2)$-limited word $\ww'$ such that $\ww\approx\ww'$ is satisfied by $\vV$.
\end{proof}

\subsection{The variety $\vR$ and its subvarieties}
\label{subsec: var R}

Let $\vR$ denote the monoid variety given by the identities~\eqref{xxyy=yyxx} and
\begin{equation}
\label{xzxyxty=xzyxty}
xzxyxty\approx xzyxty.
\end{equation}
Put $\vE=\var\{x^2\approx x^3,\,x^2y\approx xyx,\,\eqref{xxyy=yyxx}\}$. If $\vV$ is a monoid variety, then we denote by $\overleftarrow{\vV}$ the variety \textit{dual to} $\vV$, i.e., the variety consisting of monoids antiisomorphic to monoids from $\vV$. For any $n\ge1$, we put
$$
\we_n=
\begin{cases}
x&\text{if }n\text{ is odd},\\
y&\text{if }n\text{ is even}.
\end{cases}
$$

The following statement establishes restrictions on the type of identities that can be
used to define some subvarieties of $\vR$.

\begin{lemma}[\!{\cite[Lemmas~3.4 and~3.5]{Gusev-21}}] 
\label{L: R subvarieties}
Each subvariety of $\vR$ containing $\vF\vee\overleftarrow{\vE}$ is defined by the identities~\eqref{xxyy=yyxx} and~\eqref{xzxyxty=xzyxty} together with some of the following identities: \eqref{xyzxy=yxzxy} 
\begin{align}
\label{yxxtxy=xyxtxy}
yx^2txy&\approx xyxtxy,\\
\label{xxytxy=xyxtxy}
x^2ytxy&\approx xyxtxy,\\
\label{xyxztx=xyxzxtx}
xyxztx&\approx xyxzxtx,\\
\notag
\alpha_n:\enskip xy\prod_{i=1}^{n+1} (t_i \we_i)&\approx yx\prod_{i=1}^{n+1} (t_i \we_i),\\
\notag
\beta_n:\enskip yx^2\prod_{i=2}^{n+1} (t_i \we_i)&\approx xyx\prod_{i=2}^{n+1} (t_i \we_i),\\
\notag
\gamma_n:\enskip x^2y\prod_{i=1}^{n+1} (t_i \we_i)&\approx xyx\prod_{i=1}^{n+1} (t_i \we_i),\\
\notag
\gamma_n':\enskip x^2y\prod_{i=2}^{n+1} (t_i \we_i)&\approx xyx\prod_{i=2}^{n+1} (t_i \we_i),
\end{align}
where $n\ge1$.\qed
\end{lemma}

\section{The variety $\vP$}
\label{sec: var P}

Put $\vP=\var\{\eqref{xyx=xyxx},\,\eqref{xxyy=yyxx},\,\eqref{xyzxy=yxzxy},\,\beta_1,\,\gamma_1'\}$. The main goal of this section is to verify that $\vP$ and so $\overleftarrow{\vP}$ are almost Cross varieties.

\begin{proposition}
\label{P: P is almost Cross}
The variety $\vP$ is a non-finitely generated almost Cross subvariety of $\Acom$. The lattice $L(\vP)$ has the form shown in Fig.~\ref{pic: L(P)}, where
$$
\begin{aligned}
&\vC=\var\{x^2\approx x^3,\,xy\approx yx\},\\
&\vD=\var\{x^2\approx x^3,\,x^2y\approx xyx\approx yx^2\},\\
&\vH=\vP\{\eqref{xyxztx=xyxzxtx}\},\\
&\vP_n=\vP\{\alpha_n\},\ \ n\ge1.
\end{aligned}
$$
\end{proposition}

\begin{figure}[htb]
\unitlength=1mm
\linethickness{0.4pt}
\begin{center}
\begin{picture}(55,115)
\put(35,5){\circle*{1.33}}
\put(35,15){\circle*{1.33}}
\put(35,25){\circle*{1.33}}
\put(35,35){\circle*{1.33}}
\put(45,45){\circle*{1.33}}
\put(25,45){\circle*{1.33}}
\put(35,55){\circle*{1.33}}
\put(15,55){\circle*{1.33}}
\put(25,65){\circle*{1.33}}
\put(45,65){\circle*{1.33}}
\put(15,75){\circle*{1.33}}
\put(35,75){\circle*{1.33}}
\put(25,85){\circle*{1.33}}
\put(25,95){\circle*{1.33}}
\put(25,105){\circle*{1.33}}
\put(25,115){\circle*{1.33}}
\put(25,112){\makebox(0,0)[cc]{$\vdots$}}

\put(35,5){\line(0,1){30}}
\put(35,35){\line(-1,1){20}}
\put(35,35){\line(1,1){10}}
\put(25,45){\line(1,1){10}}
\put(25,45){\line(1,1){20}}
\put(15,55){\line(1,1){20}}
\put(45,45){\line(-1,1){30}}
\put(45,65){\line(-1,1){20}}
\put(15,75){\line(1,1){10}}
\put(25,85){\line(0,1){23}}

\put(35,2){\makebox(0,0)[cc]{$\vT$}}
\put(37,15){\makebox(0,0)[lc]{$\vSL$}}
\put(37,25){\makebox(0,0)[lc]{$\vC$}}
\put(37,35){\makebox(0,0)[lc]{$\vD$}}
\put(23,45){\makebox(0,0)[rc]{$\vE$}}
\put(13,55){\makebox(0,0)[rc]{$\vF$}}
\put(47,45){\makebox(0,0)[lc]{$\overleftarrow{\vE}$}}
\put(47,65){\makebox(0,0)[lc]{$\vQ$}}
\put(37,75){\makebox(0,0)[lc]{$\vP_1$}}
\put(27,85){\makebox(0,0)[lc]{$\vP_2$}}
\put(27,95){\makebox(0,0)[lc]{$\vP_3$}}
\put(27,105){\makebox(0,0)[lc]{$\vP_4$}}
\put(27,115){\makebox(0,0)[lc]{$\vP$}}
\put(13,75){\makebox(0,0)[rc]{$\vH$}}
\end{picture}
\end{center}
\caption{The lattice $L(\vP)$}
\label{pic: L(P)}
\end{figure}

\begin{proof}
Let $\vV$ be a proper subvariety of $\vP$. If $\vV$ is completely regular, then $\vV\in\{\vT,\vSL\}$ by Observation~\ref{O: cr in Acom}. So, we may assume that $\vV$ is not completely regular. If $\vQ\nsubseteq \vV$, then $\vV$ satisfies the identity 
\begin{equation}
\label{xxyzxx=xxyxzxx}
x^2yzx^2\approx x^2yxzx^2
\end{equation}
by Lemma~\ref{L: does not contain Q}.
Then the identities
$$
xyxztx\stackrel{\eqref{xyx=xyxx}}\approx xyx^2ztx^2\stackrel{\eqref{xxyzxx=xxyxzxx}}\approx  xyx^2zxtx^2\stackrel{\eqref{xyx=xyxx}}\approx xyxzxtx
$$
hold in $\vV$, whence $\vV\subseteq\vH$. The lattice $L(\vH)$ has the form shown in Fig.~\ref{pic: L(P)}~\cite[Proposition~3.1]{Gusev-20}. So, we may assume that $\vQ\subseteq\vV$. If $\vF\nsubseteq \vV$, then $\vV\subseteq\vQ$ by Lemma~\ref{L: does not contain F}. The results of \cite[Section~5]{Lee-14} imply that the lattice $L(\vQ)$ is as shown in Fig.~\ref{pic: L(P)}. Thus, we may assume that $\vF\vee\vQ\subseteq \vV$. Clearly, $\vP$ satisfies the identities~\eqref{xyzxy=yxzxy}, \eqref{yxxtxy=xyxtxy}, \eqref{xxytxy=xyxtxy}, $\beta_n$, $\gamma_n$ and $\gamma_n'$ for any $n\ge1$. Then Lemma~\ref{L: R subvarieties} and the fact that $\vQ$ violates~\eqref{xyxztx=xyxzxtx} imply that $\vV=\vP_n$ for some $n\ge1$. It follows that $\vF\vee\vQ=\vP_1$. The variety $\vH$ violates the identity $\alpha_1$ by~\cite[Lemma 3.6(i)]{Gusev-20} and satisfies the identities
$$
xyhxsytx\stackrel{\eqref{xyxztx=xyxzxtx}}\approx xyhxsxytx \stackrel{\eqref{xyzxy=yxzxy}}\approx yxhxsxytx\stackrel{\eqref{xyxztx=xyxzxtx}}\approx yxhxsytx,
$$
whence $\vH\vee\vQ=\vP_2$. Thus, the lattice $L(\vP_2)$ has the form shown in Fig.~\ref{pic: L(P)}. 

To complete the description of $L(\vP)$ it remains to verify that $\vP_n\subset\vP_{n+1}$ for any $n\ge1$. Clearly, $\vP_n\subseteq\vP_{n+1}$ for any $n\ge1$ because $\alpha_{n+1}$ follows from $\alpha_n$. So, it suffices to establish that if $\vP_{n+1}$ satisfies an identity $\wu =xy\prod_{i=1}^{n+1} (t_i \we_i^{k_i})\approx\wv$, then $\wv= xy\prod_{i=1}^{n+1} (t_i \we_i^{\ell_i})$ for some $\ell_1,\ell_2,\dots, \ell_{n+1}\ge1$. Put 
$$
\Psi=\{\eqref{xyx=xyxx},\,\eqref{xxyy=yyxx},\,\eqref{xyzxy=yxzxy},\,\alpha_{n+1},\,\beta_1,\,\gamma_1'\}.
$$
By induction, we can reduce our considerations to the case when $\{\wu,\wv\}=\{\wa\xi(\ws)\wb,\wa\xi(\wt)\wb\}$ for some words $\wa,\wb\in\sX^\ast$, an endomorphism $\xi$ of $\sX^\ast$ and $\ws\approx\wt\in\Psi$. We may assume without loss of generality that the words $\wu$ and $\wv$ are different.{\sloppy

}
\begin{observation}
\label{subwords of u}
Let $a$ and $b$ be different letters. If $ab$ is a subword of $\wu$, then this subword has exactly one occurrence in $\wu$ and $\wu$ does not contain the subword $ba$.\qed
\end{observation}

If $\xi(x)$ is the empty word, then $\xi(\ws)=\xi(\wt)$, but this is impossible because $\wu\ne\wv$. Thus, $\xi(x)$ is non-empty. Then, since $\con(\xi(x))\subseteq\mul(\xi(\ws))\subseteq\mul(\wu)$, Observation~\ref{subwords of u} implies that $\xi(x)=c^k$ for some $k\ge1$ and $c\in\{x,y\}$.

The identity~\eqref{xyx=xyxx} allows us to add and delete the occurrences of the letter $x$ next to the non-first occurrence of this letter. This implies the required conclusion whenever $\ws\approx\wt$ coincides with~\eqref{xyx=xyxx}. Suppose now that $\ws\approx\wt\in\Psi\setminus\{\eqref{xyx=xyxx}\}$. Note that $\xi(y)$ is non-empty because the identity $\wu\approx\wv$ is non-trivial. Then $\xi(y)=d^r$ for some letter $d$ and some $r\ge1$ by Observation~\ref{subwords of u}. Since the words $\wu$ and $\wv$ are different, $c\ne d$. So, $\{c,d\}=\{x,y\}$. Now we apply Observation~\ref{subwords of u} again and obtain that $\ws\approx\wt$ does not equal to~\eqref{xyzxy=yxzxy}. Further, $\ws\approx\wt\notin\{\eqref{xxyy=yyxx},\,\beta_1,\,\gamma_1'\}$ because $\wu$ does not contain the subwords $c^{2k}d^r$, $d^rc^{2k}$ and $c^kd^rc^k$. So, it remains to consider the case when the identity $\ws\approx\wt$ coincides with the identity $\alpha_{n+1}$. Evidently, there is $j\in\{1,2,\dots,n+2\}$ such the $t_i\notin\con(\xi(t_j))$ for any $i\in\{1,2,\dots,n+1\}$. It follows that $\con(\xi(t_j))\subseteq\{x,y\}$. Then $\wu$ contains either at least two different occurrences of some subwords of the form $xy$ or $yx$. A contradiction with Observation~\ref{subwords of u}. 

So, we have proved that the lattice $L(\vP)$ has the form shown in Fig.~\ref{pic: L(P)}. This fact and Lemmas~\ref{L: LF+small=FG} and~\ref{L: xyx=xyxx is locally finite} imply that any proper subvariety of $\vP$ is Cross. Then $\vP$ is almost Cross. Finally, $\vP$ is non-finitely generated by Lemmas~\ref{L: LF+small=FG} and~\ref{L: xyx=xyxx is locally finite}.
\end{proof}

If $\ww\in\sX^\ast$ and $X\subseteq\sX$, then $\ww_X$ denote the word obtained from $\ww$ by deleting all letters from $X$.

\begin{lemma}
\label{L: does not contain P_n}
Let $k\ge1$ and $\vV$ be a monoid variety that satisfies the identity~\eqref{x^n=x^{n+1}} for some $n\ge1$. 
Suppose that $\vF\vee\vQ\subseteq\vV$.
Then $\vV$ satisfies the identity
$$
\delta_{k,n}:\enskip xy\prod_{i=1}^{k+1} (t_i \we_i^n)\approx yx\prod_{i=1}^{k+1} (t_i \we_i^n)
$$
if and only if $\vP_{k+1}\nsubseteq\vV$.
\end{lemma}

\begin{proof}
For any $n\ge1$, $\vP_{k+1}\{\delta_{k,n}\}=\vP_{k+1}\{\alpha_k\}=\vP_k$ because
$$
xy\prod_{i=1}^{k+1} (t_i \we_i)\stackrel{\eqref{xyx=xyxx}}\approx xy\prod_{i=1}^{k+1} (t_i \we_i^n)\stackrel{\delta_{k,n}}\approx yx\prod_{i=1}^{k+1} (t_i \we_i^n)\stackrel{\eqref{xyx=xyxx}}\approx yx\prod_{i=1}^{k+1} (t_i \we_i).
$$
This fact and Proposition~\ref{P: P is almost Cross} imply that if $\vV$ satisfies $\delta_{k,n}$, then $\vP_{k+1}\nsubseteq\vV$.

Conversely, suppose that $\vP_{k+1}\nsubseteq\vV$. Then there exists an identity $\wu\approx\wv$ that is satisfied by the variety $\vV$ but not by the variety $\vP_{k+1}$. The variety $\vV$ is non-completely regular and non-commutative because $\vF\vee\vQ\subseteq\vV$. Let~\eqref{decomposition of u} be the decomposition of $\wu$. Lemma~\ref{L: equivalent decompositions} implies that the decomposition of $\wv$ has the form~\eqref{decomposition of v}. 

As in~\cite{Gusev-Vernikov-18}, a letter $x$ is said to be 1-\textit{divider of a word} $\ww$ if the first and the second occurrences of this letter in $\ww$ lie in different blocks. Equivalently, $x$ is a 1-divider of a word $\ww$ if $h_1(\ww,x)\ne h_2(\ww,x)$. For any $0\le i\le m$, let $X_i$ denote the set of all 1-dividers of $\wu$ in the block $\wu_i$ and $Y_i=\con(\wu_i)\setminus X_i$. The inclusion $\vF\vee\vQ\subseteq\vV$ and Lemmas~\ref{L: word problem F} and~\ref{L: word problem Q} imply that $X_i$ coincides with the set of all 1-dividers of $\wv$ in the block $\wv_i$ and $Y_i=\con(\wv_i)\setminus X_i$. Let $Y_i=\{y_{i1},y_{i2},\dots,y_{ir_i}\}$. It is routinely checked that $\vP_{k+1}$ satisfies $\wu\approx\prod_{i=0}^m(t_i\wu_i')$ and $\wv\approx\prod_{i=0}^m(t_i\wv_i')$ where $\wu_i'=(\wu_i)_{Y_i}\wh_i$,  $\wv_i'=(\wv_i)_{Y_i}\wh_i$ and $\wh_i=y_{i1}^2y_{i2}^2\cdots y_{ir_i}^2$.

For any $s=0,1,\dots,m+1$, we put $\ww_s=\prod_{i=0}^{s-1}(t_i\wv_i')\prod_{i=s}^m(t_i\wu_i')$. Since $\wu\approx\wv$ does not hold in $\vP_{k+1}$, there is $j\in\{0,1,\dots,m\}$ such that $\vP_{k+1}$ violates $\ww_j\approx\ww_{j+1}$. 
Clearly, all the letters of both $(\wu_j)_{Y_j}$ and $(\wv_j)_{Y_j}$ are simple in these words. By induction, we may assume that there are $x,y$ such that the first occurrence of $x$ precedes the first occurrence of $y$ in $\wu_j$ but $(\wv_j)_{Y_j}=\wa yx\wb$ and $\vP_{k+1}$ violates 
\begin{equation}
\label{identity that does not hold in P_{k+1}}
\ww_{j+1}\approx \Bigl(\prod_{i=0}^{j-1}(t_i\wv_i')\Bigr)\bigl(t_j\wa xy\wb\wh_j\bigr)\Bigl(\prod_{i=j+1}^m(t_i\wu_i')\Bigr).
\end{equation}

If $x,y\in\con(\wu_\ell')$ for some $\ell>j$, then we obtain a contradiction with the fact that~\eqref{identity that does not hold in P_{k+1}} is not satisfied by $\vP_{k+1}$ because
$$
\begin{aligned}
&\ww_{j+1}\stackrel{\{\eqref{xyx=xyxx},\eqref{xxyy=yyxx}\}}\approx\Bigl(\prod_{i=0}^j(t_i\wv_i')\Bigr)\Bigl(\prod_{i=j+1}^{\ell-1}(t_i\wu_i')\Bigr)\bigl(t_\ell\wu_\ell'xy\bigr)\Bigl(\prod_{i=\ell+1}^m(t_i\wu_i')\Bigr)\stackrel{\eqref{xyzxy=yxzxy}}\approx\\
&\Bigl(\prod_{i=0}^{j-1}(t_i\wv_i')\Bigr)\bigl(t_j\wa xy\wb\wh_j\bigr)\Bigl(\prod_{i=j+1}^{\ell-1}(t_i\wu_i')\Bigr)\bigl(t_\ell\wu_\ell'xy\bigr)\Bigl(\prod_{i=\ell+1}^m(t_i\wu_i')\Bigr)\stackrel{\{\eqref{xyx=xyxx},\eqref{xxyy=yyxx}\}}\approx \\
&\Bigl(\prod_{i=0}^{j-1}(t_i\wv_i')\Bigr)\bigl(t_j\wa xy\wb\wh_j\bigr)\Bigl(\prod_{i=j+1}^m(t_i\wu_i')\Bigr).
\end{aligned}
$$
So, it remains to consider the case when $|\con(\wu_\ell')\cap \{x,y\}|\le 1$ for any $\ell>j$. Then we may assume without any loss that there exists a subsequence $k_1,k_2,\dots, k_r=m+1$ of $j+1,j+2,\dots,m+1$ such that
\begin{itemize}
\item $\con(\wu_{j+1}\wu_{j+2}\cdots\wu_{k_1-1})\cap \{x,y\}=\emptyset$;\\
\item if $s\le r-1$, then $\con(\wu_{k_s}\wu_{k_s+1}\cdots\wu_{k_{s+1}-1})\cap \{x,y\}=\con(\we_s)$.
\end{itemize}
Clearly, $r>2$ because $x,y\in\con(\prod_{i=j+1}^m\mathbf u_i')=\con(\prod_{i=j+1}^m\mathbf v_i')$. If $r>k+1$, then we get a contradiction with the fact that~\eqref{identity that does not hold in P_{k+1}} is not satisfied by $\vP_{k+1}$ because this identity follows from $\alpha_k$. Now, for any $i=1,2,\dots,r-1$, we substitute $t_i\we_i^n$ for $t_{k_i}$ and~1 for all the other letters occurring in the identity $\wu\approx \wv$ except $x$ and $y$. We obtain the identity
$$
xy\prod_{i=1}^{r-1}(t_i\we_i^{p_i})\approx yx\prod_{i=1}^{r-1}(t_i\we_i^{q_i})
$$
for some $p_1,q_1,p_2,q_2,\dots,p_{r-1},q_{r-1}>n$. Then, since~\eqref{x^n=x^{n+1}} holds in $\vV$, we obtain that $\vV$ satisfies $\delta_{r-2,n}$.
Thus, we may assume that $r\le k+1$. It remains to note that $\delta_{r-2,n}$ implies $\delta_{k,n}$.
\end{proof}

\section{The variety $\vO$}
\label{sec: var O}
Recall that $\vO=\var\{\eqref{xtyzxy=xtyzyx},\,\eqref{xtxyzy=xtyxzy}\}$. Consider an identity of the form~\eqref{u_0t_0...t_mu_m=v_0t_0...t_mv_m} where $\simple(\wu)=\simple(\wv)=\{t_1,t_2,\dots,t_m\}$ and $\wu_0, \wv_0, \wu_1,\wv_1,\dots, \wu_m, \wv_m$ are words not containing any of the letters $t_1,t_2,\dots, t_m$. If $(\wu_0, \wv_0),(\wu_1, \wv_1),\dots,(\wu_m, \wv_m)\ne (\emptyset,\emptyset)$, then this identity is said to be \textit{efficient}. For any $i\ge2$, the word $\ww$ is said to be $i$-\textit{free} if, for any $\wa_1, \wa_2, \wa_3\in\sX^\ast$, the equality $\ww=\wa_1(\wa_2)^i\wa_3$ implies that $\wa_2$ is the empty word.  For any variety $\vV$ and any identity system $\Sigma$, we put $\vV\{\Sigma\}=\vV\wedge \var\,\Sigma$.  Let
$$
\vL=\vO\{x^2\approx x^3,\, x^2y\approx yx^2,\,\alpha_1\}.
$$

The main result of this section is the following description of aperiodic almost Cross subvarieties of $\vO$.

\begin{proposition}
\label{P: aperiodic almost Cross in O}
The varieties $\vL$, $\vP$ and only they are aperiodic almost Cross subvarieties of $\vO$.
\end{proposition}

\begin{proof}
Let $\vV$ be an aperiodic subvariety of $\vO$ that does not contain $\vL$ and $\vP$. We need to verify that $\vV$ is a Cross variety. In view of Observation~\ref{O: cr and comm are Cross}, we may assume that $\vV$ is non-commutative and non-completely regular. Since $\vV$ is aperiodic, it satisfies the identity $x^s\approx x^{s+1}$ for some $s\ge1$. According to~\cite[Lemma~4.4]{Lee-14}, $\vV$ satisfies $\kappa_{s',j}$ for some $0\le j\le s'$. If we multiply the identity $\kappa_{s',j}$ by $t_{s'+1}x$ on the left, then we obtain the identity $\kappa_{s'+1,j}$. Evidently, the identity $x^s\approx x^{s+1}$ implies the identity $x^{s+1}\approx x^{s+2}$. It follows that we may assume without any loss that $\vV$ satisfies the identities~\eqref{x^n=x^{n+1}} and $\kappa_{n,j}$ for some $n\ge1$ and $0\le j\le n$. 

It is proved in~\cite[Lemma~7]{Lee-13} that any small subvariety of $\vO$ is Cross. In view of this fact, it suffices to verify that $\vV$ is small. Let $\wu\approx\wv$ be an identity that does not hold in $\vV$. We verify that there is an identity $\wu'\approx\wv'$ such that $\vV\{\wu\approx\wv\}=\vV\{\wu'\approx\wv'\}$ and $\ell(\wu'),\ell(\wv')\le 50n^2$. Then any subvariety of $\vV$ can be given by identities from some fixed finite system of identities. The proof of the fact that $\vV$ is small and so the proof of Proposition~\ref{P: aperiodic almost Cross in O} are thus complete.

The identity~\eqref{x^n=x^{n+1}} allows us to assume that $\wu$ and $\wv$ are $(n+1)$-free. In view of~\cite[Lemma~8 and Remark~11]{Lee-13}, we may assume without any loss that one of the following claims holds:
\begin{itemize}
\item[(a)] $\wu \approx \wv$ coincides with the efficient identity
\begin{equation}
\label{rigid identity} 
x^{e_0} \prod_{i=1}^m (t_ix^{e_i}) \approx x^{f_0} \prod_{i=1}^m (t_ix^{f_i}) \end{equation}
for some $m\ge0$, $0\le e_0,f_0,e_1,f_1,\dots,e_m,f_m\le n$;
\item[(b)] $\wu \approx \wv$ coincides with an efficient identity 
\begin{equation}
\label{x^ey^f..=y^ex^f..}
x^{e_0}y^{f_0}\prod_{i=1}^m(t_i x^{e_i}y^{f_i})\approx y^{f_0}x^{e_0}\prod_{i=1}^m(t_i x^{e_i}y^{f_i})
\end{equation} 
for some $m\ge0$, $e_0,f_0\ge1$, $0\le e_1,f_1,e_2,f_2,\dots,e_m,f_m\le n$ and $\sum_{i=0}^m e_i,\sum_{i=0}^m f_i\ge2$.{\sloppy

}
\end{itemize}
Further considerations are divided into two cases corresponding to the claims~(a) and~(b).

\medskip

\textit{The claim}~(a) \textit{holds}. Suppose that $\vQ\nsubseteq\vV$. Lemma~\ref{L: equivalent decompositions} and Corollary~\ref{C: w is equivalent to (2n+2)-limited word} imply that there are $(2n+2)$-limited words $\wu_1=x^{e_0'}\prod_{i=1}^m(t_ix^{e_i'})$ and $\wv_1=x^{f_0'}\prod_{i=1}^m(t_ix^{f_i'})$ such that $\wu\approx\wu_1$ and $\wv\approx\wv_1$ hold in $\vV$. Let $X=\{t_i\mid 1\le i\le m \text{ and } e_i'=f_i'=0\}$. Put $\wu'=(\wu_1)_X$ and $\wv'=(\wv_1)_X$. Clearly, the identity $\wu'\approx\wv'$ is equivalent to $\wu\approx\wv$ in $\vV$. Since $\wu'\approx\wv'$ is efficient and $\wu'$, $\wv'$ are $(2n+2)$-limited, both $\wu'$ and $\wv'$ have at most $(4n+3)$ simple letters. Hence $\ell(\wu'),\ell(\wv')\le 6n+5<50n^2$. So, we may assume below that $\vQ\subseteq\vV$.

Since the identity $\wu\approx\wv$ is efficient, Lemma~\ref{L: word problem Q} implies that $e_i,f_i>0$ for any $i=0,1,\dots,m$. If $m\le2n$, then $\ell(\wu),\ell(\wv)\le n+2n(n+1)<50n^2$ because $\wu$ and $\wv$ are $(n+1)$-free. Therefore, we may assume without any loss that $m>2n$. 
Then $\vV$ satisfies the identities $\wu\stackrel{\kappa_{n,j}}\approx\wu_1$ and $\wv\stackrel{\kappa_{n,j}}\approx\wv_1$ where
$$
\begin{aligned}
\wu_1={}&x^{e_0}\Bigl(\prod_{i=1}^{n-1}(t_ix^{e_i})\Bigr)\Bigl(\prod_{i=n}^{m-n}(t_ix^n)\Bigr)\Bigl(\prod_{i=m-n+1}^m(t_ix^{e_i})\Bigr),\\
\wv_1={}&x^{f_0}\Bigl(\prod_{i=1}^{n-1}(t_ix^{f_i})\Bigr)\Bigl(\prod_{i=n}^{m-n}(t_ix^n)\Bigr)\Bigl(\prod_{i=m-n+1}^m(t_ix^{f_i})\Bigr).
\end{aligned}
$$

Now Observation~\ref{O: equivalent identities} applies with the conclusion that $\wu_1\approx\wv_1$ is equivalent to $\wu'\approx\wv'$ in $\vV$ where
$$
\begin{aligned}
\wu'={}&x^{e_0}\Bigl(\prod_{i=1}^{n-1}(t_ix^{e_i})\Bigr)\bigl(t_nx^n\bigr)\Bigl(\prod_{i=m-n+1}^m(t_ix^{e_i})\Bigr),\\
\wv'={}&x^{f_0}\Bigl(\prod_{i=1}^{n-1}(t_ix^{f_i})\Bigr)\bigl(t_nx^n\bigr)\Bigl(\prod_{i=m-n+1}^m(t_ix^{f_i})\Bigr).
\end{aligned}
$$
Then $\vV\{\wu\approx\wv\}=\vV\{\wu_1\approx\wv_1\}=\vV\{\wu'\approx\wv'\}$.
Since $\wu'$ and $\wv'$ are $(n+1)$-free, 
$$
\ell(\wu'),\ell(\wv')\le n+(n^2-1)+(n+1)+n(n+1)<50n^2,
$$
and we are done.

\medskip

\textit{The claim}~(b) \textit{holds}. 
Suppose that $\vQ\nsubseteq\vV$. Lemma~\ref{L: equivalent decompositions}, Corollary~\ref{C: w is equivalent to (2n+2)-limited word} and its proof imply that there are $(2n+2)$-limited words $\wu_1=x^{e_0'}y^{f_0'}\prod_{i=1}^m(t_ix^{e_i'}y^{f_i'})$ and $\wv_1=y^{f_0'}x^{e_0'}\prod_{i=1}^m(t_ix^{e_i'}y^{f_i'})$ such that $\wu\approx\wu_1$ and $\wv\approx\wv_1$ hold in $\vV$. Let $X=\{t_i\mid 1\le i\le m \text{ and } e_i'=f_i'=0\}$. Put $\wu'=(\wu_1)_X$ and $\wv'=(\wv_1)_X$. Clearly, the identity $\wu'\approx\wv'$ is equivalent to $\wu\approx\wv$ in $\vV$. Since $\wu'\approx\wv'$ is efficient and $\wu'$, $\wv'$ are $(2n+2)$-limited, both $\wu'$ and $\wv'$ have at most $(4n+2)$ simple letters. Hence $\ell(\wu'),\ell(\wv')\le 8n+6<50n^2$. So, we may assume below that $\vQ\subseteq\vV$. Then $(e_i,f_i)\ne(0,0)$ for any $i=0,1,\dots,m$ because the identity $\wu\approx\wv$ is efficient. If $m\le8n$, then $\ell(\wu),\ell(\wv)\le 2n+8n(2n+1)<50n^2$ because $\wu$ and $\wv$ are $(n+1)$-free. Therefore, we may assume without any loss that $m>8n$. Then either $e=\sum_{i=0}^me_i>2n$ or $f=\sum_{i=0}^mf_i>2n$ because $(e_i,f_i)\ne(0,0)$ for any $i=0,1,\dots,m$. By symmetry, we may assume that $e>2n$. Further considerations are divided into two cases: $f\le 2n$ or $f>2n$.

\smallskip

\textit{Case} 1: $f\le 2n$. There are $\ell\le 2n$ and a maximal subsequence $k_1,k_2,\dots, k_\ell$ of $4n,2n+1,\dots,m-4n$ such that $f_{k_1},f_{k_2},\dots, f_{k_{\ell}}>0$. For convenience, we put $k_{\ell+1}=m-4n+1$. Then $\wu=x^{e_0}y^{f_0}\ww$ and $\wv=y^{f_0}x^{e_0}\ww$ where
$$
\ww=\prod_{i=1}^{4n-1}(t_ix^{e_i}y^{f_i})\prod_{i=4n}^{k_1-1}(t_ix^{e_i})\prod_{s=1}^{\ell}\bigl((t_{k_s}x^{e_{k_s}}y^{f_{k_s}})\prod_{i=k_s+1}^{k_{s+1}-1}(t_ix^{e_i})\bigr)\prod_{i=k_{\ell+1}}^m(t_ix^{e_i}y^{f_i}).
$$
Then $\prod_{i=0}^{4n-1}e_i,\prod_{i=m-4n+1}^me_i>n$ because the identity $\wu\approx\wv$ is efficient and $f\le 2n$.  It follows that $\vV$ satisfies the identities $\wu\stackrel{\kappa_{n,j}}\approx\wu_1$ and $\wv\stackrel{\kappa_{n,j}}\approx\wv_1$ where $\wu_1=x^{e_0}y^{f_0}\ww_1$, $\wv_1=y^{f_0}x^{e_0}\ww_1$ and
$$
\ww_1=\prod_{i=1}^{4n-1}(t_ix^{e_i}y^{f_i})\prod_{i=4n}^{k_1-1}(t_ix^n)\prod_{s=1}^{\ell}\bigl((t_{k_s}x^{e_{k_s}}y^{f_{k_s}})\prod_{i=k_s+1}^{k_{s+1}-1}(t_ix^n)\bigr)\!\prod_{i=k_{\ell+1}}^m(t_ix^{e_i}y^{f_i}).
$$

Now Observation~\ref{O: equivalent identities} applies with the conclusion that $\wu_1\approx\wv_1$ is equivalent to $\wu'\approx\wv'$ in $\vV$ where $\wu'=x^{e_0}y^{f_0}\ww'$, $\wv'=y^{f_0}x^{e_0}\ww'$ and
$$
\ww'=\Bigl(\prod_{i=1}^{4n-1}(t_ix^{e_i}y^{f_i})\Bigr)\bigl(t_{4n}x^n\bigr)\Bigl(\prod_{s=1}^{\ell}\bigl((t_{k_s}x^{e_{k_s}}y^{f_{k_s}})(t_{k_s+1}x^n)\bigr)\Bigr)\Bigl(\prod_{i=m-4n+1}^m(t_ix^{e_i}y^{f_i})\Bigr).
$$
Then $\vV\{\wu\approx\wv\}=\vV\{\wu_1\approx\wv_1\}=\vV\{\wu'\approx\wv'\}$. Since $\wu'$ and $\wv'$ are $(n+1)$-free and $\ell\le 2n$, 
$$
\ell(\wu'),\ell(\wv')\le 2n+(4n-1)(2n+1)+(n+1)+2n(3n+2)+4n(2n+1)\le 50n^2,
$$
and we are done.
\smallskip

\textit{Case} 2: $f>2n$. Then $\vV\{\wu\approx\wv\}$ satisfies the identity
\begin{equation}
\label{x^{e_0}y^{f_0}tx^ny^n=y^{f_0}x^{e_0}tx^ny^n}
x^{e_0}y^{f_0}tx^ny^n\approx y^{f_0}x^{e_0}tx^ny^n
\end{equation}
because
$$
\begin{aligned}
&x^{e_0}y^{f_0}tx^ny^n\stackrel{\eqref{x^n=x^{n+1}}}\approx x^{e_0}y^{f_0}tx^{n+e}y^{n+f}\stackrel{\eqref{xtyzxy=xtyzyx}}\approx x^{e_0}y^{f_0}tx^ny^n\prod_{i=1}^m(x^{e_i}y^{f_i})\stackrel{\eqref{x^ey^f..=y^ex^f..}}\approx\\
&y^{f_0}x^{e_0}tx^ny^n\prod_{i=1}^m(x^{e_i}y^{f_i})\stackrel{\eqref{xtyzxy=xtyzyx}}\approx y^{f_0}x^{e_0}tx^{n+e}y^{n+f}\stackrel{\eqref{x^n=x^{n+1}}}\approx y^{f_0}x^{e_0}tx^ny^n.
\end{aligned}
$$

Suppose that $\vF\nsubseteq\vV$. Then $\vV$ satisfies~\eqref{xyx^n=x^nyx^n} by Lemma~\ref{L: does not contain F}. Since $e,f>2n$, there are $c$ and $d$ such that $e_c,f_d>0$ and $n\le\sum_{i=1}^ce_i, \sum_{i=c}^me_i,\sum_{i=1}^df_i, \sum_{i=d}^mf_i$. By symmetry, we may assume that $c\le d$. Then $\vV\{\wu\approx\wv\}=\vV\{\eqref{xyx^n=x^nyx^n},\,\eqref{x^{e_0}y^{f_0}tx^ny^n=y^{f_0}x^{e_0}tx^ny^n}\}$ because
$$
\begin{aligned}
&\wu\stackrel{\kappa_{n,j}}\approx x^{e_0}y^{f_0}\ww\stackrel{\eqref{xyx^n=x^nyx^n}}\approx
x^{n-1+e_0}y^{n-1+f_0}\ww\stackrel{\eqref{x^n=x^{n+1}}}\approx
x^{n+e_0}y^{n+f_0}\ww\stackrel{\eqref{xtxyzy=xtyxzy}}\approx
x^{e_0}y^{f_0}x^ny^n\ww\stackrel{\eqref{x^{e_0}y^{f_0}tx^ny^n=y^{f_0}x^{e_0}tx^ny^n}}\approx\\
&y^{f_0}x^{e_0}x^ny^n\ww\stackrel{\eqref{xtxyzy=xtyxzy}}\approx
y^{n+f_0}x^{n+e_0}\ww\stackrel{\eqref{x^n=x^{n+1}}}\approx
y^{n-1+f_0}x^{n-1+e_0}\ww\stackrel{\eqref{xyx^n=x^nyx^n}}\approx
y^{f_0}x^{e_0}\ww\stackrel{\kappa_{n,j}}\approx
\wv,
\end{aligned}
$$
where
$$
\ww=\Bigl(\prod_{i=1}^{c-1}(t_ix^{e_i}y^{f_i})\Bigr)\bigl(t_cx^ny^{f_c}\bigr)\Bigl(\prod_{i=c+1}^{d-1}(t_ix^{e_i}y^{f_i})\Bigr)\bigl(t_dx^{e_d}y^n\bigr)\Bigl(\prod_{i=d+1}^m(t_ix^{e_i}y^{f_i})\Bigr),
$$
and we are done. So, we may assume below that $\vF\subseteq\vV$. Then $\vV$ satisfies $\delta_{s,s'}$ for some $s,s'\ge1$ by Proposition~\ref{P: P is almost Cross}, Lemma~\ref{L: does not contain P_n} and the fact that $\vP\nsubseteq\vV$. Clearly, $\delta_{s,s'}$ implies $\delta_{s+1,s'}$ and $\delta_{s,s'+1}$. Then the arguments quite analogous to ones from the first paragraph of the proof of this proposition allow us to assume that $\vV$ satisfies $\delta_{n,n}$.

Since $\wu\approx\wv$ is efficient, either $\sum_{i=0}^{4n-1}e_i\ge n$ or $\sum_{i=0}^{4n-1}f_i\ge n$ and either $\sum_{i=m-4n+1}^me_i\ge n$ or $\sum_{i=m-4n+1}^mf_i\ge n$. By symmetry, we may assume that $\sum_{i=0}^{4n-1}e_i\ge n$. Further considerations are divided into two cases: $\sum_{i=m-4n+1}^me_i\ge n$ or $\sum_{i=m-4n+1}^mf_i\ge n$.

\smallskip

\textit{Case} 2.1: $\sum_{i=m-4n+1}^me_i\ge n$. Put
$$
Y_1=\{s\mid 4n\le s\le m-4n,\ n\le\sum_{i=0}^sf_i\} \text{ and } Y_2=\{s\mid 4n\le s\le m-4n,\ n\le\sum_{i=s}^mf_i\}.
$$
Let 
$$ 
p = \begin{cases} \min Y_1 & \text{if }Y_1\ne\emptyset, \\ m-4n & \text{otherwise}; \end{cases} 
\ \text{ and }\ 
q = \begin{cases} \max Y_2 & \text{if }Y_2\ne\emptyset, \\ 4n & \text{otherwise}. \end{cases}
$$
Clearly, $p\le q$ because $f>2n$. There exist $\ell,r\le n$ and a maximal subsequence $k_1,k_2,\dots, k_\ell$ of $4n,4n+1,\dots,p-1$ such that $f_{k_1},f_{k_2},\dots, f_{k_{\ell}}>0$, and a maximal subsequence $h_1,h_2,\dots, h_r$ of $q+1,q+2,\dots,m-4n$ such that $f_{h_1},f_{h_2},\dots, f_{h_r}>0$. For convenience, we put $k_{\ell+1}=p$ and $k_{r+1}=m-4n+1$. Then
$$
\begin{aligned}
\prod_{i=4n}^{p-1}(t_ix^{e_i}y^{f_i})={}&\prod_{i=4n}^{k_1-1}(t_ix^{e_i})\prod_{s=1}^{\ell}\bigl((t_{k_s}x^{e_{k_s}}y^{f_{k_s}})\prod_{i=k_s+1}^{k_{s+1}-1}(t_ix^{e_i})\bigr)\\
\prod_{i=q+1}^{m-4n}(t_ix^{e_i}y^{f_i})={}&\prod_{i=q+1}^{h_1-1}(t_ix^{e_i})\prod_{s=1}^r\bigl((t_{h_s}x^{e_{h_s}}y^{f_{h_s}})\prod_{i=h_s+1}^{h_{s+1}-1}(t_ix^{e_i})\bigr).
\end{aligned}
$$
For any $i=p,p+1,\dots,q$, put
$$ 
e_i' = \begin{cases} n & \text{if }e_i>0, \\ 0 & \text{otherwise}; \end{cases} 
\ \text{ and }\ 
f_i' = \begin{cases} n & \text{if }f_i>0, \\ 0 & \text{otherwise}. \end{cases}
$$
Then $\vV$ satisfies the identities $\wu\stackrel{\kappa_{n,j}}\approx\wu_1$ and $\wv\stackrel{\kappa_{n,j}}\approx\wv_1$ where $\wu_1=x^{e_0}y^{f_0}\ww_1\ww_2\ww_3$, $\wv_1=y^{f_0}x^{e_0}\ww_1\ww_2\ww_3$, $\ww_2=\prod_{i=p}^q(t_ix^{e_i'}y^{f_i'})$ and
$$
\begin{aligned}
\ww_1={}&\prod_{i=1}^{4n-1}(t_ix^{e_i}y^{f_i})\prod_{i=4n}^{k_1-1}(t_ix^n)\prod_{s=1}^{\ell}\bigl((t_{k_s}x^{e_{k_s}}y^{f_{k_s}})\prod_{i=k_s+1}^{k_{s+1}-1}(t_ix^n)\bigr),\\
\ww_3={}&\prod_{i=q+1}^{h_1-1}(t_ix^n)\prod_{s=1}^r\bigl((t_{h_s}x^{e_{h_s}}y^{f_{h_s}})\prod_{i=h_s+1}^{h_{s+1}-1}(t_ix^n)\bigr)\prod_{i=m-4n+1}^m(t_ix^{e_i}y^{f_i}).
\end{aligned}
$$
If $e_s'=f_s'$ for some $p\le s\le q$, then $e_s'=f_s'=n$ because $\wu\approx\wv$ is efficient. Then $\vV\{\wu\approx\wv\}=\vV\{\eqref{x^{e_0}y^{f_0}tx^ny^n=y^{f_0}x^{e_0}tx^ny^n}\}$ because 
$\wu\stackrel{\kappa_{n,j}}\approx\wu_1\stackrel{\eqref{x^{e_0}y^{f_0}tx^ny^n=y^{f_0}x^{e_0}tx^ny^n}}\approx \wv_1\stackrel{\kappa_{n,j}}\approx\wv$, and we are done. So, we may assume that $(e_i',f_i')\in\{(n,0),(0,n)\}$ for any $i=p,p+1,\dots,q$. By symmetry, we may assume that there exist $b\ge0$ and a subsequence $g_1=p,g_2,\dots, g_{b+1}=q+1$ of $p,p+1,\dots,q+1$ such that $\ww_2=\prod_{i=1}^b\bigl(\prod_{s=g_i}^{g_{i+1}-1}(t_s\we_i^n)\bigr)$.
If $b>n$, then $\vV\{\wu\approx\wv\}=\vV$ because 
$\wu\stackrel{\kappa_{n,j}}\approx\wu_1\stackrel{\delta_{n,n}}\approx \wv_1\stackrel{\kappa_{n,j}}\approx\wv$, and we are done. Thus, we may assume that $b\le n$. 

Now Observation~\ref{O: equivalent identities} applies with the conclusion that $\wu_1\approx\wv_1$ is equivalent to $\wu'\approx\wv'$ where $\wu'=x^{e_0}y^{f_0}\ww_1'\ww_2'\ww_3'$, $\wv'=y^{f_0}x^{e_0}\ww_1'\ww_2'\ww_3'$, $\ww_2'=\prod_{i=1}^b(t_{g_i}\we_i^n)$ and
$$
\begin{aligned}
\ww_1'={}&\Bigl(\prod_{i=1}^{4n-1}(t_ix^{e_i}y^{f_i})\Bigr)\bigl(t_{2n}x^n\bigr)\Bigl(\prod_{s=1}^{\ell}(t_{k_s}x^{e_{k_s}}y^{f_{k_s}}t_{k_s+1}x^n)\Bigr),\\
\ww_3'={}&\bigl(t_{q+1}x^n\bigr)\Bigl(\prod_{s=1}^r(t_{h_s}x^{e_{h_s}}y^{f_{h_s}}t_{h_s+1}x^n)\Bigr)\Bigl(\prod_{i=m-4n+1}^m(t_ix^{e_i}y^{f_i})\Bigr).
\end{aligned}
$$
Then $\vV\{\wu\approx\wv\}=\vV\{\wu_1\approx\wv_1\}=\vV\{\wu'\approx\wv'\}$.
Since $\wu'$, $\wv'$ are $(n+1)$-free and $\ell,r,b\le n$, we have that
$$
\begin{aligned}
&\ell(\ww_1')\le (4n-1)(2n+1)+(n+1)+n(3n+2)<20n^2,\\
&\ell(\ww_2')\le n(n+1)\le2n^2,\\
&\ell(\ww_3')\le (n+1)+n(3n+2)+4n(2n+1)<20n^2.
\end{aligned}
$$
It follows that $\ell(\wu'),\ell(\wv')<50n^2$, and we are done.

\smallskip

\textit{Case} 2.2: $\sum_{i=m-4n+1}^mf_i\ge n$. Put
$$
Y_1=\{s\mid 4n\le s\le m-4n,\ n\le\sum_{i=0}^sf_i\} \text{ and } Y_2=\{s\mid 4n\le s\le m-4n,\ n\le\sum_{i=s}^me_i\}.
$$
Let 
$$ 
p = \begin{cases} \min Y_1 & \text{if }Y_1\ne\emptyset, \\ m-4n & \text{otherwise}; \end{cases} 
\ \text{ and }\ 
q = \begin{cases} \max Y_2 & \text{if }Y_2\ne\emptyset, \\ 4n & \text{otherwise}. \end{cases}
$$

\textit{Case} 2.2.1: $q<p$. Then there exist $\ell,r\le n$ and a maximal subsequence $k_1,k_2,\dots, k_\ell$ of $4n,4n+1,\dots,q$ such that $f_{k_1},f_{k_2},\dots, f_{k_{\ell}}>0$, and a maximal subsequence $h_1,h_2,\dots, h_r$ of $p,p+1,\dots,m-4n$ such that $e_{h_1},e_{h_2},\dots, e_{h_r}>0$. For convenience, we put $k_{\ell+1}=q+1$ and $k_{r+1}=m-4n+1$. Then
$$
\begin{aligned}
\prod_{i=4n}^{q}(t_ix^{e_i}y^{f_i})={}&\prod_{i=4n}^{k_1-1}(t_ix^{e_i})\prod_{s=1}^{\ell}\bigl((t_{k_s}x^{e_{k_s}}y^{f_{k_s}})\prod_{i=k_s+1}^{k_{s+1}-1}(t_ix^{e_i})\bigr)\\
\prod_{i=p}^{m-4n}(t_ix^{e_i}y^{f_i})={}&\prod_{i=p}^{h_1-1}(t_iy^{f_i})\prod_{s=1}^r\bigl((t_{h_s}x^{e_{h_s}}y^{f_{h_s}})\prod_{i=h_s+1}^{h_{s+1}-1}(t_iy^{f_i})\bigr).
\end{aligned}
$$
Then $\vV$ satisfies the identities $\wu\stackrel{\kappa_{n,j}}\approx\wu_1$ and $\wv\stackrel{\kappa_{n,j}}\approx\wv_1$ where $\wu_1=x^{e_0}y^{f_0}\ww_1\ww_2\ww_3$, $\wv_1=y^{f_0}x^{e_0}\ww_1\ww_2\ww_3$, $\mathbf w_2=\prod_{i=q+1}^{p-1}(t_ix^{e_i}y^{f_i})$ and
$$
\begin{aligned}
\ww_1={}&\prod_{i=1}^{4n-1}(t_ix^{e_i}y^{f_i})\prod_{i=4n}^{k_1-1}(t_ix^n)\prod_{s=1}^{\ell}\bigl((t_{k_s}x^{e_{k_s}}y^{f_{k_s}})\prod_{i=k_s+1}^{k_{s+1}-1}(t_ix^n)\bigr),\\
\ww_3={}&\prod_{i=p+1}^{h_1-1}(t_iy^n)\prod_{s=1}^r\bigl((t_{h_s}x^{e_{h_s}}y^{f_{h_s}})\prod_{i=h_s+1}^{h_{s+1}-1}(t_iy^n)\bigr)\prod_{i=m-4n+1}^m(t_ix^{e_i}y^{f_i}).
\end{aligned}
$$

Now Observation~\ref{O: equivalent identities} applies with the conclusion that $\wu_1\approx\wv_1$ is equivalent to $\wu'\approx\wv'$ where $\wu'=x^{e_0}y^{f_0}\ww_1'\ww_2\ww_3'$, $\wv'=y^{f_0}x^{e_0}\ww_1'\ww_2\ww_3'$ and
$$
\begin{aligned}
\ww_1'={}&\Bigl(\prod_{i=1}^{4n-1}(t_ix^{e_i}y^{f_i})\Bigr)\bigl(t_{4n}x^n\bigr)\Bigl(\prod_{s=1}^{\ell}(t_{k_s}x^{e_{k_s}}y^{f_{k_s}}t_{k_s+1}x^n)\Bigr),\\
\ww_3'={}&\bigl(t_{p+1}y^n\bigr)\Bigl(\prod_{s=1}^r(t_{h_s}x^{e_{h_s}}y^{f_{h_s}}t_{h_s+1}y^n)\Bigr)\Bigl(\prod_{i=m-4n+1}^m(t_ix^{e_i}y^{f_i})\Bigr).
\end{aligned}
$$
Then $\vV\{\wu\approx\wv\}=\vV\{\wu_1\approx\wv_1\}=\vV\{\wu'\approx\wv'\}$. Since $\wu'$, $\wv'$ are $(n+1)$-free and $\ell,r\le n$, we have that
$$
\begin{aligned}
&\ell(\ww_1')\le (4n-1)(2n+1)+(n+1)+n(3n+2)<20n^2,\\
&\ell(\ww_3')\le (n+1)+n(3n+2)+4n(2n+1)<20n^2.
\end{aligned}
$$
The definitions of $p$ and $q$ imply that $\ell(\ww_2)\le4n$. It follows that $\ell(\wu'),\ell(\wv')<50n^2$, and we are done.

\smallskip

\textit{Case} 2.2.2: $p\le q$. This case is considered similarly to Case~2.1. 

\medskip

So, we have proved that each subvariety of $\vV$ can be given by identities from some fixed finite system of identities both hand-sides of which are of length~$\le50n^2$. It follows that the variety $\vV$ is small and so Cross by aforementioned Lemma~7 of~\cite{Lee-13}.
\end{proof}

\section{Main result}
\label{sec: main result}

To formulate the main result of the article we put
$$
\begin{aligned}
&\vK=\var\{\eqref{xyx=xyxx},\,\eqref{xxy=xxyx},\,\eqref{xxyy=yyxx}\},\\
&\vJ=\var
\left\{
\left.
\begin{array}{l}
\eqref{xyx=xyxx},\,\eqref{xxyy=yyxx},\,\eqref{xyzxy=yxzxy},\,\eqref{xyxztx=xyxzxtx},\\
xz_{1\pi} z_{2\pi}\cdots z_{n\pi}x\biggl(\,\prod\limits_{i=1}^nt_iz_i\biggr)\approx
x^2z_{1\pi} z_{2\pi}\cdots z_{n\pi}\biggl(\,\prod\limits_{i=1}^nt_iz_i\biggr)
\end{array}
\right|
\begin{array}{l}
n\ge1,\\
\pi\in S_n
\end{array}
\right\}
\end{aligned}
$$
where $S_n$ denote the full symmetric group on the set $\{1,2,\dots,n\}$. Let $\vM$ and $\vN$ denote the monoid varieties generated by the monoids $S(xzxyty)$ and $S(xyzxty,xtyzxy)$, respectively.{\sloppy

}
The main result of the article is the following

\begin{theorem}
\label{T: main result}
A subvariety of $\Acom$ is Cross if and only if it excludes the varieties $\vJ$, $\overleftarrow{\vJ}$ $\vK$, $\overleftarrow{\vK}$, $\vL$, $\vM$, $\vN$, $\vP$ and $\overleftarrow{\vP}$. Consequently, the class $\Acom$ contains precisely nine almost Cross subvarieties.
\end{theorem}

\begin{proof}
The varieties listed in Theorem~\ref{T: main result} are almost Cross:
\begin{itemize}
\item $\vJ$ and $\overleftarrow{\vJ}$ by~\cite[Corollary~1.2 and the dual to it]{Gusev-20};
\item $\vK$ and $\overleftarrow{\vK}$ by~\cite[Proposition~6.1 and the dual to it]{Gusev-Vernikov-18};
\item $\vL$ by~\cite[Theorem~2]{Lee-13}; 
\item $\vM$ and $\vN$ by~\cite[Proposition~5.1]{Jackson-05}; 
\item $\vP$ and $\overleftarrow{\vP}$ by Proposition~\ref{P: P is almost Cross} and the dual to it.
\end{itemize}
Therefore, every Cross subvariety of $\Acom$ does not contain these varieties.
So, to complete the proof it remains to verify that if $\vV$ is a subvariety of $\Acom$ and $\vV$ excludes the varieties listed in Theorem~\ref{T: main result}, then $\vV$ is Cross.

Suppose that $S(xyx)\in\vV$. Since $\vM\nsubseteq\vV$ and $\vN\nsubseteq\vV$, either $\vV\subseteq\vO$ or $\vV\subseteq\overleftarrow{\vO}$ by Lemma~\ref{L: isoterm} and~\cite[Fact~3.1]{Sapir-15}. Then Proposition~\ref{P: aperiodic almost Cross in O} and the dual to it imply that $\vV$ is a Cross variety because $\vV$ does not contain the varieties $\vL$, $\vP$ and $\overleftarrow{\vP}$.

Suppose now that $S(xyx)\notin\vV$. In view of Observation~\ref{O: cr and comm are Cross}, we assume that $\vV$ is non-commutative and non-completely regular. According to Lemma~\ref{L: isoterm}, $\vV$ satisfies a non-trivial identity of the form $xyx\approx \ww$. Then $\ww =x^pyx^q$ for some $p$ and $q$ such that $p\ge 2$ or $q\ge 2$ by Lemma~\ref{L: equivalent decompositions}. By symmetry, we may assume that $q\ge 2$. 

Suppose at first that $p=0$. Then $\vV$ satisfies the identity $x^2\approx x^q$ and, therefore, the identity $xyx\approx yx^2$. It follows from~\cite[the dual to Corollary~3.6]{Lee-14} that every periodic variety that satisfies the latter identity is Cross. Thus, $\vV$ is a Cross variety. Suppose now that $p\ge 1$.  Then $\vV$ satisfies the identity $x^2\approx x^{p+q}$ and, therefore, the identity $x^2\approx x^3$ because $\vV$ is aperiodic. Then the identities
$$
xyx\approx x^pyx^q\approx x^pyx^{q+1}\approx xyx^2
$$
hold in $\vV$. Thus, \eqref{xyx=xyxx} is satisfied by $\vV$. In view of Observation~\ref{O: var in Acom}, $\vV$ satisfies the identity~\eqref{x^ny^n=y^nx^n} for some $n\ge1$. This identity together with~\eqref{xyx=xyxx} implies~\eqref{xxyy=yyxx}. 

Consider an arbitrary subvariety $\vX$ of $\vV$. If $\overleftarrow{\vE}\nsubseteq\vX$, then $\vX\subseteq\vK$ by~\cite[Lemma~2.3(i)]{Gusev-21}. Since $\vV\ne\vK$ and $\vK$ is an almost Cross variety, we obtain that $\vX$ is a Cross subvariety of $\vK$. If $\vF\nsubseteq\vX$, then $\vX\subseteq\vQ$ by Lemma~\ref{L: does not contain F}. In view of~\cite[Proposition~5.4]{Lee-14}, $\vX$ is a Cross variety again.

In particular, it remains to consider the case when $\vF\vee\overleftarrow{\vE}\subseteq\vV$. Since $\vV\ne\vJ$, \cite[Lemma~4.1]{Gusev-21} implies that $\vV$ satisfies the identity \eqref{xzxyxty=xzyxty}. Hence $\vV\subseteq\vR$. Further, suppose that $\vQ\nsubseteq\vV$. Then $\vV$ satisfies the identity
\begin{equation}
\label{xxyzx=xxyxzx}
x^2yzx\approx x^2yxzx
\end{equation} 
by Lemma~\ref{L: does not contain Q}. Clearly, this identity together with~\eqref{xyx=xyxx} imply the identity~\eqref{xyxztx=xyxzxtx}. It is easy the see that every identity of the form $\mathbf pt_1\mathbf e_1\cdots t_k\mathbf e_k\approx \mathbf qt_1\mathbf e_1\cdots t_k\mathbf e_k$ is equivalent modulo~$\{\eqref{xyx=xyxx},\,\eqref{xxyzx=xxyxzx}\}$ to the identity
\[
\mathbf p\cdot t_1\mathbf e_1t_2\mathbf e_2t_3\mathbf e_3t_4\mathbf e_4\cdot t_5\cdots t_{k-1}\mathbf e_{k-1}t_k\mathbf e_k\approx \mathbf q\cdot t_1\mathbf e_1t_2\mathbf e_2t_3\mathbf e_3t_4\mathbf e_4\cdot t_5\cdots t_{k-1}\mathbf e_{k-1}t_k\mathbf e_k.
\] 
Evidently, the latter identity is equivalent modulo~\eqref{xyx=xyxx} to
\[
\mathbf p\cdot t_1\mathbf e_1t_2\mathbf e_2t_3\mathbf e_3t_4\mathbf e_4\cdot t_{k-1}\mathbf e_{k-1}t_k\mathbf e_k\approx \mathbf q\cdot t_1\mathbf e_1t_2\mathbf e_2t_3\mathbf e_3t_4\mathbf e_4\cdot t_{k-1}\mathbf e_{k-1}t_k\mathbf e_k.
\] 
This fact and Lemma~\ref{L: R subvarieties} imply that each subvariety of $\vV$ containing $\vF\vee\overleftarrow{\vE}$ may be given by some identities from the identity system 
$$
\{\eqref{xxyy=yyxx},\,\eqref{xyzxy=yxzxy},\,\eqref{xzxyxty=xzyxty},\, \eqref{yxxtxy=xyxtxy},\,\eqref{xxytxy=xyxtxy},\,\eqref{xyxztx=xyxzxtx},\,\alpha_r,\,\beta_r,\,\gamma_r,\,\gamma_r'\mid r\le 6\}.
$$
The evident fact that this identity system is finite and the previous paragraph imply that the variety $\vV$ is small and finitely based. Now Lemmas~\ref{L: LF+small=FG} and~\ref{L: xyx=xyxx is locally finite} apply with the conclusion that $\vV$ is finitely generated. So, we have proved that $\vV$ is a Cross variety. Let now $\mathbf Q\subseteq\mathbf V$. Then, since $\vV\ne\vP$, Proposition~\ref{P: P is almost Cross} implies that $\vP_{k+1}\nsubseteq\vV$ for some $k\ge0$.  Lemma~\ref{L: does not contain P_n} applies with the conclusion that $\vV$ satisfies $\delta_{k,2}$. Since~\eqref{xyx=xyxx} holds in $\vV$, the identity $\delta_{k,2}$ is equivalent to the identity $\alpha_k$ in $\vV$. Evidently, for any $r>k$, the identities $\alpha_r$, $\beta_r$, $\gamma_r$ and $\gamma_r'$ follow from $\alpha_k$. This fact and Lemma~\ref{L: R subvarieties} imply that each subvariety of $\vV$ containing $\vF\vee\overleftarrow{\vE}$ may be given by some identities from the identity system 
$$
\{\eqref{xxyy=yyxx},\,\eqref{xyzxy=yxzxy},\,\eqref{xzxyxty=xzyxty},\, \eqref{yxxtxy=xyxtxy},\,\eqref{xxytxy=xyxtxy},\,\eqref{xyxztx=xyxzxtx},\,\alpha_r,\,\beta_r,\,\gamma_r,\,\gamma_r'\mid r\le k\}.
$$
Since this identity system is finite, the same arguments as above imply that $\vV$ is a Cross variety again.
\end{proof}

\end{document}